\documentclass[10pt,twoside]{article}
\usepackage{amsfonts,amssymb,amsmath,amsthm,hyperref,zb2013}
\newtheorem{thm}{Theorem}
\newtheorem{prop}[thm]{Proposition}
\newtheorem{lem}[thm]{Lemma}

\theoremstyle{definition}
\newtheorem{defn}[thm]{Definition}

\providecommand{\norm}[2][\relax]{\left\|#2\right\|\ifx#1\relax\else_{#1}\fi}
\providecommand{\modulus}[2][\relax]{\left| #2 \right|\ifx#1\relax\else_{#1}\fi}
\providecommand{\scalar}[3][\relax]{\left\langle #2,#3 
        \right\rangle\ifx#1\relax\else_{#1}\fi}
\providecommand{\map}[1]{\mathsf{#1}}
\providecommand{\uir}[3][0]{\ifcase #1{\rho^{#2}_{#3}}%
\or {\breve{\rho}^{#2}_{#3}}%
\or {\tilde{\rho}^{#2}_{#3}}%
\or {\hat{\rho}^{#2}_{#3}}\fi}

\providecommand{\FSpace}[3][]{\ensuremath{\ifx#2l \ell_{#3}^{#1}{}\else
    #2_{#3}^{#1}{}\fi}} 
\providecommand{\oper}[1]{\mathcal{#1}}

\providecommand{\Zbl}[1]{Zbl\href{http://www.emis.de:80/cgi-bin/zmen/ZMATH/en/zmathf.html?first=1&maxdocs=3&type=html&an=#1&format=complete}{#1}}
\providecommand{\myeprint}[2]{E-print: \href{#1}{\texttt{#2}}}

\DeclareFontFamily{U}{mathx}{\hyphenchar\font45}
\DeclareFontShape{U}{mathx}{m}{n}{
      <5> <6> <7> <8> <9> <10>
      <10.95> <12> <14.4> <17.28> <20.74> <24.88>
      mathx10
      }{}
\DeclareSymbolFont{mathx}{U}{mathx}{m}{n}
\DeclareFontSubstitution{U}{mathx}{m}{n}
\DeclareMathAccent{\wideparen}{0}{mathx}{"75}
\begin{document}

\let\oldenddocument\enddocument
\let\document\begingroup
\def\enddoc{\endgroup\setcounter{equation}{0}\setcounter{footnote}{0}
\vspace{0pt plus 1fil}}
\let\enddocument=\enddoc

\renewcommand{\ehkol}{V.V. Kisil}
\renewcommand{\ohkol}{Boundedness of Relative Convolutions on
Nilpotent Lie Groups} \Vykh{kisil:FirstPage}{kisil:LastPage}

\renewcommand{\thefootnote}{}
\footnote{On  leave from Odessa University.}
\footnote{\hfill\copyright \ \ {\bf V.V.~Kisil, 2013}}

\UDC{UDC\  517.9}

\Author{\Large Vladimir V.~Kisil}

\Address{(School of Mathematics,
University of Leeds,
Leeds LS2\,9JT,
UK)}

\Title{Boundedness \\of Relative Convolutions\\ on
Nilpotent Lie Groups \label{kisil:FirstPage}}
\noindent {\small\bf  kisil@maths.leeds.ac.uk}

\vskip 2mm

\centerline{\small\sl Dedicated to memory of Professor Promarz
M.~Tamrazov}

\vskip 2mm

\AbsEng{We discuss some norm estimations for integrated
  representations. We use the covariant transform to extend Howe's
  method from the Heisenberg group to general nilpotent Lie groups.}

\medskip

\setcounter{page}{172}

\thispagestyle{empty}

\noindent {\bf 1. Introduction.}   
%
Let \(G\) be a locally compact group, a left-invariant (Haar) measure on \(G\) is
denoted by \(dg\).  Let \(\uir{}{}\) be a bounded representation of the group
\(G\) in a vector space \(V\). The representation can be extended to a
function \(k\in\FSpace{L}{1}(G,dg)\) though integration:
\begin{equation}
  \label{eq:integrated-rep}
  \uir{}{}(k)=\int_{G} k(g)\,\uir{}{}(g)\,dg.
\end{equation} 
It is a homomorphism of the convolution algebra
\(\FSpace{L}{1}(G,dg)\) to an algebra of bounded operators on \(V\).

There are many important classes of operators described
by~\eqref{eq:integrated-rep}, notably pseudodifferential operators
(PDO) and Toeplitz
operators~\cite{Howe80b,Kisil94e,Kisil98a,Kisil13a}.  Thus, it is
important to have various norm estimations of \(\uir{}{}(k)\). We
already mentioned a straightforward inequality
\(\norm{\uir{}{}(k)}\leq C \norm[1]{k}\) for
\(k\in\FSpace{L}{1}(G,dg)\), however, other classes are of interest as
well.

If \(G\) is the Heisenberg group and \(\uir{}{}\) is its Schr\"odinger
representation, then \(\uir{}{}(\hat{a})\) is a PDO \(a(X,D)\) with
the symbol \(a\)~\cite{Howe80b,Folland89,Kisil13a}. Here, \(\hat{a}\)
is the Fourier transform of \(a\), as usual. The
Calder\'on--Vaillancourt theorem~\cite[Ch.~XIII]{MTaylor81} estimates \(\norm{a(X,D)}\) by
\(\FSpace{L}{\infty}\)-norm of a finite number of partial derivatives
of \(a\).

In this paper we revise the method used in~\cite[\S~3.1]{Howe80b} to
prove the Calder\'on--Vaillancourt estimations. It was described as ``rather
magical'' in~\cite[\S~2.5]{Folland89}. We hope, that  a usage of the covariant
transform dispel the mystery without undermining the power of the method. 
\vskip 0.2cm

\noindent {\bf 2. Preliminaries.}   
%
Through the paper \(G\) denotes an exponential Lie group.  For a
square integrable irreducible representation \(\uir{}{}\) of \(G\) in
a Hilbert space \(V\) and a fixed admissible mother wavelet \(\phi\in
V\), the wavelet transform \(\oper{W}_\phi: V \rightarrow
\FSpace{C}{b}(G)\) is~\cite{AliAntGaz00,Kisil98a,Kisil13a}:
\begin{equation}
  \label{eq:wavelet-transform}
   [\oper{W}_\phi
   v](g):=\scalar{\uir{}{}(g^{-1})v}{\phi}=\scalar{v}{\uir{}{}(g)\phi},\qquad g\in G,\
   v\in V. 
\end{equation}
For an unimodular \(G\), the left \(\Lambda(g): f(g')\mapsto
f(g^{-1}g') \) and the right \(R: f(g')\mapsto f(g'g)\) regular
representations of \(G\) are unitary operators on
\(\FSpace{L}{2}(G,dg)\).
The covariant transforms intertwines the left and right regular
representations of \(G\) with the following actions of \(\uir{}{}\):
\begin{equation}
  \label{eq:covariant-intertwine}
  \Lambda  (g) \oper{W}_\phi = \oper{W}_\phi \uir{}{}(g)
  \quad\text{ and }\quad
  R (g) \oper{W}_\phi= \oper{W}_{ \uir{}{}(g) \phi }
  \quad\text{ for all } g\in G.
\end{equation}
 
For a fixed admissible vector \(\psi\in V\), the integrated
representation~\eqref{eq:integrated-rep} produces the contravariant
transform \(\oper{M}_\psi: \FSpace{L}{1}(G) \rightarrow V\),
cf.~\cite{Kisil98a,Kisil13a}: 
\begin{equation}
  \label{eq:contravariant-transform}
  \oper{M}_\psi^{\uir{}{}} (k)= \uir{}{}(k)\psi, \qquad
  \text{ where } k\in \FSpace{L}{1}(G).
\end{equation}
The contravariant transform \(\oper{M}_{\psi}^{\uir{}{}}\) intertwines the
left regular representation \( \Lambda  \)
on \( \FSpace{L}{2}(G)\) and \( \uir{}{} \):
\begin{equation}
  \label{eq:inv-transform-intertwine}
  \oper{M}_{\psi}^{\uir{}{}}\, \Lambda (g) = \uir{}{}(g)\, \oper{M}_{\psi}^{\uir{}{}}.
\end{equation}
Combining with~\eqref{eq:covariant-intertwine}, we see that the
composition \(\oper{M}_\psi^{\uir{}{}} \circ
\oper{W}_\phi^{\uir{}{}}\) of the covariant and contravariant
transform intertwines \(\uir{}{}\) with itself.  For an irreducible
square integrable \(\uir{}{}\) and suitably normalised admissible
\(\phi\) and \(\psi\), we use the Schur's lemma
\cite[Lem.~4.3.1]{AliAntGaz00}, \cite[Thm.~8.2.1]{Kirillov76} to
conclude that:
\begin{equation}
  \label{eq:wave-trans-inverse}
  \oper{M}_\psi^{\uir{}{}} \circ
  \oper{W}_\phi^{\uir{}{}}=\scalar{\psi}{\phi} I.
\end{equation}  

Let \(H\) be a subgroup of \(G\) and \(X=G/H\) be the respective
homogeneous space (the space of left cosets) with a (quasi-)invariant
measure \(dx\)~\cite[\S~9.1]{Kirillov76}. There is the natural
projection \(\map{p}: G \rightarrow X\). We usually fix a continuous
section \(\map{s}: X \rightarrow G\)~\cite[\S~13.2]{Kirillov76}, which
is a right inverse to \(\map{p}\).
We also define an operator of relative convolution on \(V\)~\cite{Kisil94e,Kisil13a},
cf.~\eqref{eq:integrated-rep}:
\begin{equation}
    \label{eq:relative-conv}
    \uir{}{}(k)=\int_{X} k(x)\,\uir{}{}(\map{s}(x))\,dx,
\end{equation}
with a kernel \(k\) defined on \(X=G/H\).  
\vskip 0.2cm

\noindent {\bf 3. Norm Estimations.}
%
We start from the following lemma, which has a transparent proof in
terms of covariant transform, cf. ~\cite[\S~3.1]{Howe80b}
and~\cite[(2.75)]{Folland89}. For the rest of the paper we assume that
\(\uir{}{}\) is an irreducible square integrable representation of an
exponential Lie group \(G\) in \(V\) and mother wavelet \(\phi,\psi
\in V\) are admissible.
\begin{lem}
  Let \(\phi\in V\) be such that, for \(\Phi=\oper{W}_\phi \phi\), the
  reciprocal \(\Phi^{-1}\) is bounded on \(G\) or \(X=G/H\).  Then,
  for the integrated representation~\eqref{eq:integrated-rep} or
  relative convolution~\eqref{eq:relative-conv}, we have the
  inequality:
  \begin{equation}
    \label{eq:norm-inequal-repres}
    \norm{\uir{}{}(f)}\leq \norm{\Lambda \otimes R(f \Phi^{-1})},
  \end{equation}
  where \((\Lambda \otimes R)(g): k(g')\mapsto k(g^{-1}g'g)\)
  acts on the image of \(\oper{W}_\phi\).
\end{lem}
\begin{proof}
  We know from~\eqref{eq:wave-trans-inverse} that \(\oper{M}_\phi
  \circ \oper{W}_{\uir{}{}(g)\phi} =
  \scalar{\phi}{\uir{}{}(g)\phi} I\) on \(V\), thus:
  \begin{displaymath}
    \oper{M}_\phi \circ \oper{W}_{\uir{}{}(g)\phi} \circ \uir{}{}(g) =
    \scalar{\phi}{\uir{}{}(g)\phi} \uir{}{}(g) =\Phi (g) \uir{}{}(g). 
  \end{displaymath}
  On the other hand, the intertwining
  properties~\eqref{eq:covariant-intertwine} of the wavelet transform 
  imply:
  \begin{displaymath}
    \oper{M}_\phi \circ \oper{W}_{\uir{}{}(g)\phi} \circ \uir{}{}(g) =
    \oper{M}_\phi \circ (\Lambda \otimes R)(g)  \circ\oper{W}_{\phi}.
  \end{displaymath}
  Integrating the identity \(\Phi (g) \uir{}{}(g)=  \oper{M}_\phi
  \circ (\Lambda \otimes R)(g)  \circ\oper{W}_{\phi}\) with the
  function \(f\Phi^{-1}\) and use the partial isometries \(\oper{W}_\phi\)  and
  \(\oper{M}_\phi\) we get the inequality.
\end{proof}

The Lemma is most efficient if \(\Lambda \otimes R\) act in a simple
way. Thus, we give he following
\begin{defn}
  We say that the subgroup \(H\) has the \emph{complemented commutator
    property}, if there exists a continuous section \(\map{s}: X \rightarrow
  G\) such that:
  \begin{equation}
    \label{eq:H-upper-triangle-prop}
    \map{p}(\map{s}(x)^{-1}g \map{s}(x))=\map{p}(g),\qquad \text{ for
      all } x\in X=G/H, \ g\in G.
  \end{equation}
\end{defn}
For a Lie group \(G\) with the Lie algebra \(\mathfrak{g}\) define the
Lie algebra \(\mathfrak{h}=[\mathfrak{g},\mathfrak{g}]\). The subgroup
\(H=\exp(\mathfrak{h})\) (as well as any larger subgroup) has the
complemented commutator property~\eqref{eq:H-upper-triangle-prop}. Of
course, \(X=G/H\) is non-trivial if \(H\neq G\) and this happens, for
example, for a nilpotent \(G\).  In particular, for the Heisenberg
group, its centre has the complemented commutator property.

Note, that the complemented commutator
property~\eqref{eq:H-upper-triangle-prop} implies:
\begin{equation}
  \label{eq:h-identity-defn}
  \Lambda\otimes R(\map{s}(x)): g \mapsto gh, \quad \text{ for the
  unique } h=g^{-1}\map{s}(x)^{-1}g \map{s}(x)\in H.
\end{equation}
For a character \(\chi\) of the subgroup \(H\), we introduce an integral
transformation \(\wideparen{\ }:{L}_{1}(X)\rightarrow {C}(G)\):
\begin{equation}
  \label{eq:paren-transf}
  \wideparen{k}(g)=\int_X k(x)\, \chi(g^{-1}\map{s}(x)^{-1}g \map{s}(x))\,dx, 
\end{equation}
where \(h(x,g)=g^{-1}\map{s}(x)^{-1}g \map{s}(x)\) is in \(H\) due to
the relations~\eqref{eq:H-upper-triangle-prop}. This transformation
generalises the isotropic symbol defined for the Heisenberg group in
\cite[\S~2.1]{Howe80b}.
\begin{prop}
  Let a subgroup \(H\) of \(G\) has the complemented commutator
  property~\eqref{eq:H-upper-triangle-prop} and \(\uir{}{\chi}\) be
  an irreducible representation of \(G\) induced from a character
  \(\chi\) of \(H\), then
  \begin{equation}
    \label{eq:norm-ineq-Hn-operat}
    \norm{\uir{}{\chi}(f)}\leq 
    \norm[\infty]{\wideparen{f\Phi^{-1}}},
  \end{equation}
  with the \(\sup\)-norm of the function \(\wideparen{f\Phi^{-1}}\)
  on the right.
\end{prop}
\begin{proof}
  For an induced representation \(\uir{}{\chi}\)~\cite[\S~13.2]{Kirillov76}, the covariant
  transform \(\oper{W}_\phi\) maps  \(V\) to a space
  \(\FSpace[\chi]{L}{2}(G)\) of functions having the property
  \(F(gh)=\chi(h)F(g)\)~\cite[\S~3.1]{Kisil13a}. 
  From~\eqref{eq:h-identity-defn}, the restriction of \(\Lambda\otimes
  R\) to the space \(\FSpace[\chi]{L}{2}(G)\) is, see:
  \begin{displaymath}
        \Lambda\otimes R(\map{s}(x)): \psi(g) \mapsto \psi(gh)=\chi(h(x,g))\psi(g).
  \end{displaymath}
  In other words, \(\Lambda\otimes R\) acts by multiplication on
  \(\FSpace[\chi]{L}{2}(G)\). 
  Then, integrating the representation \(\Lambda \otimes R\) over \(X\)
  with a function \(k\) we get an operator \((L\otimes R)(k)\),
  which reduces on the irreducible component to multiplication by the
  function \(\wideparen{k}(g)\).
  Put \(k=f\Phi^{-1}\) for \(\Phi=\oper{W}_\phi \phi\). Then, from
  the inequality~\eqref{eq:norm-inequal-repres}, the norm of operator
  \(\uir{}{\chi}(f)\) can be estimated by \(\norm{\Lambda \otimes
    R(f \Phi^{-1})}=\norm[\infty]{\wideparen{f\Phi^{-1}}}\).
\end{proof}

For a nilpotent step \(2\) Lie group, the
transformation~\eqref{eq:paren-transf} is almost the Fourier
transform, cf. the case of the Heisenberg group
in~\cite[\S~2.1]{Howe80b}. This allows to estimate
\(\norm[\infty]{\wideparen{f\Phi^{-1}}}\) through
\(\norm[\infty]{\wideparen{f}}\), where \(\wideparen{f}\) is in the
essence the symbol of the respective PDO. For other groups, the
expression \(g^{-1}\map{s}(x)^{-1}g \map{s}(x)\)
in~\eqref{eq:paren-transf} contains non-linear terms and its analysis
is more difficult. In some circumstance the integral Fourier
operators~\cite[Ch.~VIII]{MTaylor81} may be useful for this purpose.

\bigskip

{\small

\providecommand{\noopsort}[1]{} \providecommand{\printfirst}[2]{#1}
  \providecommand{\singleletter}[1]{#1} \providecommand{\switchargs}[2]{#2#1}
  \providecommand{\irm}{\textup{I}} \providecommand{\iirm}{\textup{II}}
  \providecommand{\vrm}{\textup{V}} \providecommand{\cprime}{'}
  \providecommand{\eprint}[2]{\texttt{#2}}
  \providecommand{\myeprint}[2]{\texttt{#2}}
  \providecommand{\arXiv}[1]{\myeprint{http://arXiv.org/abs/#1}{arXiv:#1}}
  \providecommand{\doi}[1]{\href{http://dx.doi.org/#1}{doi:
  #1}}\providecommand{\CPP}{\texttt{C++}}
  \providecommand{\NoWEB}{\texttt{noweb}}
  \providecommand{\MetaPost}{\texttt{Meta}\-\texttt{Post}}
  \providecommand{\GiNaC}{\textsf{GiNaC}}
  \providecommand{\pyGiNaC}{\textsf{pyGiNaC}}
  \providecommand{\Asymptote}{\texttt{Asymptote}}

}

\label{kisil:LastPage}

\def \enddocument{\oldenddocument}

\end{document}